\documentclass[preprint,12pt]{amsart}

\usepackage{amssymb,amsthm,amsmath,float, xcolor}

\newtheorem{theorem}{Theorem}
\newtheorem{definition}{Definition}
\newtheorem{lemma}{Lemma}
\newtheorem{remark}{Remark}
\begin{document}

\title[I.H.F and flat circle bundles]{Intrinsically harmonic forms and flat circle bundles}

\author{Elizeu França}
\address{Universidade Federal do Amazonas \\
69103-128 Itacoatiara, AM, Brazil}
\email{elizeufranca@ufam.edu.br}

\author{Francesco Mercuri}
\address{Departament of Mathematics - IMECC, Universidade Estadual de Campinas \\
13083-970 Campinas, SP, Brazil}
\email{mercuri@ime.unicamp.br}



\begin{abstract}
We establish a criterion for the flatness of a principal circle bundle in terms of the intrinsically harmonic form problem. It states that the flatness is equivalent to the intrinsic harmonicity of a certain natural associated form. 
\end{abstract}

\maketitle

\section*{Acknowledgements}
This study was financed in part by the Coordena\c{c}\~ao de Aperfei\c{c}oamento de Pessoal de N\'ivel Superior - Brasil (CAPES).

\section*{Introduction and statement of results}

A classical theorem due to Hodge states that in each De Rham's cohomology class of a compact Riemannian manifold, there is one, and only one, harmonic form. A natural question is the following.
\begin{center}
{\em Given a closed-form $\omega$ on a compact manifold $M$, is there a Riemannian metric $g$ on $M$ such that $\omega$ is harmonic concerning $g$?}
\end{center}
If such a metric exists, $\omega$ is called an {\em intrinsically harmonic form} (I.H.). The problem of obtaining an intrinsic characterization of harmonic forms was first placed by Calabi  \cite{calabi1969intrinsic}.  
He showed  that an $1$-form $\omega$ is I.H., under suitable conditions on its zero-set,  if and only if it is \textit{transitive}. 
Transitivity means that for every point $x$ that is not a zero of $\omega$, there is an embedded circle containing $x$, such that the restriction of $\omega$ to it never vanishes. Volkov \cite{volkov2008characterization} made the most notable advance in this subject. He generalized Calabi's characterization of I.H. $1$-forms by showing that a closed 1-form is I.H. if, and only if, it is harmonic in a neighborhood of its zero set and transitive.  The condition of transitivity can be generalized for higher degree forms. Roughly speaking,  a $k$-form $\omega$ is transitive if any ``regular point'' is contained in a closed submanifold, to which  $\omega$ restricts to a volume form. In \cite{honda1997harmonic}  Honda proved a ``dual'' version of Calabi-Volkov's result by showing that a transitive closed ($n-1$)-form is I.H., under suitable conditions on its zero-set. 

By a standard foliation's argument, one can prove that nowhere-vanishing closed 1-form on a closed manifold is transitive, hence I.H. Naturally, the following issue comes up. Is a nowhere-vanishing closed $(n-1)$-form I.H.? As we will see, there is a relation between this issue and the one of obtaining conditions so that a nondegenerate volume-preserving flow admits a global cross-section. 

Given a volume-form $\Omega$ on an orientable manifold, every closed $(n-1)$-form corresponds to a unique vector field which induces a flow preserving $\Omega$. Conversely, any volume-preserving flow induces a closed $(n-1)$-form. Furthermore, the fixed points of the flow are exactly the singular points of the associate form (Lemma \ref{associated form}). In Honda's thesis, the following criterion is implicitly given

\begin{theorem}
Let $M$ be a closed orientable manifold. A nowhere-vanishing volume-preserving flow defined on $M$ admits a global cross-section if, and only if, the induced closed nowhere-vanishing $(n-1)$-form is I.H.
\end{theorem} 

Using this Theorem and a standard foliation's argument, we can prove that the existence of a transversal $C^r$-foliation ($r\geq 2$) is equivalent to the existence of a global cross-section (Theorem \ref{An intrinsic characterization of nowhere-vanishing harmonic (n-1)-forms by foliation}).

We look for examples of $(n-1)$-forms that are {\em not} I.H. among the nowhere-vanishing given in circle bundles. Since a closed nowhere-vanishing $(n-1)$-form is I.H. if, and only if, the induced flow admits a complementary foliation, the pull-back of the volume form on a circle bundle is an I.H. form if, and only if, the bundle is smoothly foliated (or admits a global cross-section).

The first criterion to decide if a flow admits a global cross-section is due to Schwartzman \cite{schwartzman1957asymptotic}. He introduced the notion of asymptotic cycles. Such a cycle is a real homology class defined for each invariant measure. A global cross-section is determined by an integral 1-dimensional cohomology class that is positive on all asymptotic cycle. With this result in the hands, we provide a class of I.H. forms, and by using such a class, we characterize flat circle bundles. Our main results are the following theorems.

\begin{theorem}\label{a class of i.h. forms}
Let $M$ be a  closed smooth manifold  with a nowhere-vanishing closed $(n-1)$-form $\omega$ inducing a pointwise periodic flow. If the orbits   of the induced flow are homologous to each other and $[\omega]\neq 0$ in $H^{n-1}_{DR}(M)$, then $\omega$ is I.H. Furthermore, there exists a smooth $\mathbb{S}^1$-action on $M$ with the same orbits as the induced flow.
\end{theorem} 

\begin{theorem}\label{main result}
 A differentiable principal circle bundle $\mathcal{B}=\{B,p,M,\mathbb{S}^1\}$ over a closed orientable base admits a flat connection, if and only if, one (and consequently all) of the four conditions below holds:
\begin{itemize}
\item[(1)] $p^*(\Omega_M)$ is I.H.;
\item[(2)] $p^*(\Omega_M)$ determines a nonzero cohomology class;
\item[(3)] the fiber represents a nonzero class in $H_1(B;\mathbb{R})$;
\item[(4)] $\chi(\mathcal{B})$, the Euler's class of $\mathcal{B}$, is a torsion element. 
\end{itemize}
\end{theorem}

Also, we presented similar results for the non-orientable cases. In degrees different from $  0, 1, n-1, n$, the question of obtaining an intrinsic characterization of harmonic forms is still open. A notable aspect of the problem, illustrating its difficulty,  is the fact that Calabi's argument does not apply to intermediate degrees (see Remark \ref{transitive p-form of rank k admits transversal forms}).

Throughout this work, all manifolds, differential forms (forms for short), flows and vector fields are considered to be $C^\infty$, unless explicitly otherwise stated.


\subsection{Hodge theory}


Let $\mathbb{V}$ be an Euclidean $n$-dimensional orientated vector space. Given $k\in\{0,1,\ldots,n\}$, the Hodge-star operator is defined as a unique operator
    \[
*:\wedge^k\mathbb{V} \longrightarrow \wedge^{n-k}\mathbb{V}
    \]
with the following property: for an oriented orthonormal basis $\{e_1,\ldots,e_n\}$ , the operator $*$ maps $e_1\wedge\ldots\wedge e_k$  to $e_{k+1}\wedge \ldots\wedge e_n$.  For a Riemannian manifold $(M,g)$ the Hodge-star operator is defined fiber-wise. Denoting by $\Omega$ the volume form on $M$ associated with the Riemannian metric on $g$, we have that

\begin{itemize}
\item[(1)] $(\omega,\eta)\rightarrow \int_M\omega\wedge *\eta$ defines an $L^2$-inner product on the space of $k$-forms;
\item[(2)] concerning this inner product, there is an $L^2$-norm on the space of $k$-forms, and $\omega\wedge *\omega=\parallel \omega\parallel \Omega$;
\item[(3)] for closed manifolds (compact and without boundary), the adjoint of the differential operator in $k$-forms is given by $d^*=(-1)^{n(p-1)+1}*d*$, where $n$ is the dimension of $M$. 
\end{itemize}

 \begin{definition}
 Let $(M,G)$ be a Riemannian manifold. The \textit{Laplace-Beltrami operator} concerning to $g$,  $\Delta_g:\Omega(M)\longrightarrow\Omega(M)$,  is defined by $\Delta_g=dd^*+d^*d$. A form $\omega$ is said  to be $g$-\textit{harmonic} provided $\Delta_g\omega=0$.
 \end{definition}

\begin{lemma}
Let $(M,g)$ be a closed Riemannian manifold and $\omega$ be a closed form on $M$. Then $\Delta_g\omega=0$ if and only if $d^*\omega=0$.
\end{lemma}

\begin{proof}
If $\omega$ is $g$-harmonic, then 
    \[
\parallel d^*\omega\parallel=(d^*\omega,d^*\omega)=(\omega,dd^*\omega+d^*\underbrace{d\omega}_{0})=(\omega,\underbrace{\Delta_g\omega}_{0})=0.
    \]
 It follows that $d^*\omega=0$. The converse is immediate.   
\end{proof}

\begin{theorem}(Hodge) In each cohomology class of a Riemannian closed manifold $M$ there exists a unique harmonic representation. 
\end{theorem}

It follows from Hodge's Theorem that if $H^k_{DR}(M)=0$, then the only harmonic  form of degree $k$ is the null-form. 

\begin{definition}
Let $M$ be a manifold. A form $\omega$ on $M$ is said to be \textit{intrinsically harmonic} (I.H.) provided $\Delta_g\omega=0$ for some Riemannian  metric $g$ on $M$. 
\end{definition}

Hear  some simple observations.

\begin{itemize}
\item[1)] An exact form is intrinsically harmonic if and only if it is the zero form.
\item[2)] A 0-form is intrinsically harmonic if and only if is locally constant.
\item[3)] A non zero $n$-form is intrinsically harmonic if and only if it never vanishes.
\item[4)] A symplectic 2-form is harmonic with respect to the associated Riemannian metric (see \cite{honda1997harmonic,manfiosymplectic}).
\end{itemize}

\subsection{Associated vector field}

Let $M$ be an orientable manifold (with boundary or not). Let $\omega$  be an $(n-1)$-form on $M$, and $\Omega$ be a volume form on $M$. The next lemma establishes an one-to-one equivalence between closed nowhere-vanishing $(n-1)$-forms and volume preserving vector fields without singularities. 

\begin{lemma}\label{associated form}
There exists a unique vector field $X$ such that $\omega=\iota_X\Omega$. The form is closed if, and only if, the flow $\Phi$ generated by $X$ is volume preserving. Furthermore, the fixed points of $\Phi$ are the same as the zeros of $\omega$.
 \end{lemma}

\begin{proof}
An $(n-1)$-form $\omega$ on $M$ defines a unique vector field $X$ such that $\iota_X\Omega = \omega$, in the following way. Let $\{U_i\}$ be an atlas for $M$, and $\{\lambda_i\}$ a partition of the unity subordinated to it. On each open $U_i$, one can choose a local frame $\{X_1^i, \cdots X_n^i\}$  so that
\[
    \begin{cases} 
        \iota_{X^i_1}\omega \equiv 0, \\
        \Omega(X_1^i, \cdots X_n^i) \equiv 1.
    \end{cases}
\]
Once such choice is made, one can see that in $U_i$ the following equality holds:
\begin{equation}\label{omega_Omega}
        \omega = \omega(X_2^i, \cdots X_n^i)\iota_{X^i_1}\Omega.
\end{equation}
We define a field $X$ on the entire $M$ by setting
\[
    X:= \sum_i\lambda_i\omega(X_2^i, \cdots X_n^i)X_1^i.
\]
It then follows from Equation \ref{omega_Omega} that
\[
    \omega = \sum_i \lambda_i\omega = \sum_i\lambda_\omega(X_2^i, \cdots X_n^i)\iota_{X_1^i}\Omega = \iota_X\Omega.
\]
Suppose now that $d\omega=0$. Then
\[
    \mathcal{L}_X\Omega = \iota_X\mathrm{d}\Omega + \mathrm{d}\iota_X\Omega = \mathrm{d}\omega = 0,
\]
so that $X$ is indeed volume preserving. Conversely, given any volume preserving field $X$, the $1$-form $\omega = \iota_X\Omega$ is closed, since by Cartan's Formula
\[
\mathrm{d}\omega = \mathrm{d}(\iota_X\Omega) = \mathcal{L}_X\Omega - \iota_X\mathrm{d}\Omega = 0.
\]
Finally,  $\omega_x=0$ if and only if $i_{X_x}\Omega_x=0$. Since $\Omega_x\neq 0$, it follows that $\omega_x=0$ if, and only if $X_x=0$. 
\end{proof}

\begin{remark}
From now on, the symbols $X_\omega$, $\omega_X$, $\Phi_X$, etc. as well as the phrases \textit{associated form}, \textit{associated vector field}, \textit{associated flow}, etc., respectively,  will be used with the meaning given by the previous theorem.  Furthermore, the 
$\mathbb{R}$-action induced by $\Phi$ will denoted by  $(x,t)\longrightarrow xt$ (the case of interest is when $M$ is compact, which implies the flow's completeness.) 

\end{remark}

\subsection{Asymptotic cycles on smooth manifolds}

Let $M$ be a compact manifold, and $C=C^\infty(M,\mathbb{S}^1)$ be the set of all smooth functions $f:M\longrightarrow \mathbb{S}^1$, on which the group structure is given by $(f\cdot g)(x)= f(x)\cdot g(x)$. Set $R$ as the following subgroup of $C$  
	\[
\{f\in C; f(x)=\exp{(2\pi i H(x))}, H:M\longrightarrow \mathbb{R} \text{ smooth}\} .		\]
Every element $f\in C$ is given locally by $f(x)=\exp 2\pi i H(x)$, where $H(x)$ is some smooth real function determined up to an additive constant. If $H_1$ and $H_2$ are determined as above in two different coordinate systems, $dH_1$ and $dH_2$ agree in the region  of overlapping. Then, for each $f\in C$, there is a closed 1-form $\omega_f$ given locally by $dH$. Consider the local parametrization of the circle given by $\theta(\exp{(2\pi it)})=t$, and the usual volume form on $\mathbb{S}^1$ described locally by $d\theta$. Let $\alpha:I\longrightarrow M$ be a smooth curve with  $X(x)=\partial_t|_0\alpha(t)$. If $f=\exp 2\pi i H(x)$ in a neighborhood of $\alpha(0)$, then it is easily seen that 
\begin{align*}
f^*d\theta(X)=\omega_f(X).
\end{align*}
Therefore $\omega_f=f^*d\theta$. In fact, the association $f\longrightarrow \omega_f$ give us  
	\[
[f]\in C/R\longrightarrow [\omega_f]\in H^1_{DR}(M;\mathbb{Z}).
	\]
a isomorphism.

\begin{definition}
A point $x\in M$ is called \textit{quasi-regular} provided 
	\[
\lim_{T\longrightarrow\infty}1/T\int_0^Tf(xt)dt
	\]
exists for every real-valued continuous function defined on $M$.
\end{definition}

\begin{remark}
The symbol $f^{'}(x)$  will  denote the derivative of $f$   in the $X(x)$-direction, that is, 
\begin{align}\label{f linha}
f^{'}(x)={\partial_t}|_{0}f(xt)=f_*(X(x)).
\end{align}
\end{remark}

\begin{definition}
Let $x$ be a quasi-regular point. The \textit{asymptotic cycle} associated with $x$ is defined as being the linear extension of $A_x:H^1(M;\mathbb{Z})\longrightarrow \mathbb{R}$ to $H^1(M;\mathbb{R})$, where $A_x$ is given by
	\[
A_x[ f]= \lim_{T\longrightarrow \infty}\frac{1}{2\pi i T}\int_{0}^{T}  \frac{f^{'}(xt)}{f(xt)}dt.
	\]
\end{definition}

\begin{definition}
A finite Borel measure $\mu$ on $M$ is said to be \textit{invariant} concerning the flow $\Phi$ provided $\int_M (f(xt)-f(x))d\mu(x)=0$ for all continuous functions $f:M\rightarrow \mathbb{R}$ and all $t\in \mathbb{R}$.
\end{definition}

\begin{definition}
The asymptotic cycle associated with an invariant measure $\mu$ is the homological class (in De Rham sense \cite[\S 18]{de1984varietes}) determined by the continuous  linear functional (\textit{current}) $A_\mu:\mathcal{D}_1\longrightarrow \mathbb{R}$ defined by
	\[
A_\mu(\omega)=\int_M\omega(X)d\mu.
	\]
\end{definition}
It is easily seen that $\partial A_\mu=0$ (that is, $\partial A_\mu$ is a cycle) hence determine an integral homological class that will be denoted too by $A_\mu$,  $A_\mu:H^1_{DR}(M;\mathbb{R})\longrightarrow \mathbb{R}$

\begin{theorem}
Let $\mu$ be an invariant measure concerning a flow $\Phi$. Let
$A^{'}_\mu:C\longrightarrow \mathbb{R}$ defined by  $A^{'}_\mu([f])=\int_MA_x[f]d\mu(x)$. Then the natural linear extension of $A^{'}_\mu$ to $H_{DR}^1(M;\mathbb{R})$ and $A_\mu$ coincides. 
\end{theorem}

\begin{proof}
By \cite[Theorem on pag. 274]{schwartzman1957asymptotic}, we have 
    \[
A^{'}_\mu([f])=\frac{1}{2\pi i }\int_M \frac{f^{'}(x)}{f(x)}d\mu.
    \]
Expressing $f$ locally by $f(x)=\exp 2\pi iH(x)$, we have 
    \[
f^{'}(x)=2\pi iH_*(\partial_t|_0(xt))f(x).
    \]
Hence 
    \[
f^{'}(x)/(2\pi i f(x))=\partial_t|_0H(xt)=\omega_f(X(x)).
    \]
It follows that
    \[
A^{'}_\mu([f])= \int_M \omega_f(X(x))d\mu=A_\mu(\omega_f).
    \]
\end{proof}

\begin{remark}
Given a quasi-regular point $x$, there exists a unique measure $\mu_x$ such that 
	\[
\lim_{T\longrightarrow\infty}1/T\int_0^Tf(xt)dt=\int_Mf(y)d\mu_x(y)
	\]
for every continuous function $f$ \cite[\S 2]{schwartzman1957asymptotic}. It follows that 
    \[
A_x(f)=\lim_{T\rightarrow \infty}\frac{1}{2\pi iT}\int_0^T\frac{f^{'}(xt)}{f(xt)}dt=\lim_{T\rightarrow \infty}\frac{1}{ T}\int_0^T\omega_f(X)dt=
    \]
    \[
\int_M\omega_f(X)(y)d_{\mu_x}(y)=A_{\mu_x}(\omega_f),    
    \]
hence $A_x=A_{\mu_x}$.
\end{remark}

\begin{lemma}\label{canonical asymptotic cycle associated with a volume-preserving flow}
Let $M$ be a closed orientable manifold. Let   $\Omega$ be a volume form on $M$ and  $X$ be a smooth vector field on $M$ preserving $\Omega$. Set $\omega=i_X\Omega$ and $\mu$ the measure determined by $\Omega$ , i.e., the measure which in an open set $B$ is given by $\mu(B)=\int_B\Omega$.  Then $\mu$ is invariant by the flow determined by $X$ and $A_\mu=[\omega]$, $\omega$ viewed as a \textit{diffuse current}\footnote{A $k$-form $\omega$ determines a current, named \textit{diffuse current}, $\omega:\mathcal{D}_{n-k}\longrightarrow\mathbb{R}$, given by $\omega(\eta)=\int_M\omega\wedge \eta$.}. In particular, if $\omega$ determines a nonzero cohomology class, then $A_\mu\neq 0$ in $H_1^{DR}(M)$.
\end{lemma}

\begin{proof}
Let  $\eta$  be a 1-form on $M$. Using the identity  $i_X(\eta\wedge\Omega)=i_X\eta\wedge\Omega+\eta\wedge i_X\Omega$ and the definition of $\mu$, one has

\begin{align*}
\int_M\eta(X)d\mu=\int_M\eta(X)\Omega=\\
\int_Mi_X\eta\wedge\Omega=\\
\int_M(i_X(\eta\wedge\Omega)+\eta\wedge i_X\Omega)=\\
\int_M\eta\wedge\omega.
\end{align*}
Thus, since $d(f\omega)=df\wedge\omega-fd\omega$ and $d\omega=0$, it follows from Stokes's theorem that  
	\[
\int_Mdf(X)d\mu=\int_Mdf\wedge\omega=\int_Md(f\omega)=0
	\]
for all $f\in C^\infty(M)$. Using this, we will prove that $\mu$ is invariant.
 Let $t\in \mathbb{R}$ and $f\in C^\infty(M)$. By Fubini's theorem

\begin{align*}
\int_M(f(xt)-f(x))d\mu(x)=\int_M\left(\int_0^tdf(X(xs))ds\right)d\mu(x)=\\
\int_0^t\left(\int_Mdf(X(xs))d\mu(x)\right)ds=0.
\end{align*}
Hence, since $\int_M(f(xt)-f(x))d\mu(x)=0$ for every $t\in \mathbb{R}$ and $f\in C^\infty(M)$, it follows that $\mu$ is invariant. Therefore, have sense $A_\mu$ and 
	\[
A_\mu(\eta)=\int_M\eta(X)d\mu=\int_M\eta\wedge\omega =\omega(\eta),
	\]
where $\omega$ is seen as a diffuse current. Now, by De Rham' theory of currents, the form $\omega$ determine a nonzero cohomology class if and only if, as a diffuse current, determine a nonzero homological class. The proof is completed.
\end{proof}

\begin{lemma}\label{asymptotic cycles of periodic points}Let $M$ be a manifold, $\Phi$ be a continuous flow on $M$, and $x \in M$ a periodic point  concerning $\Phi$ \footnote{The period function, $\lambda:M\longrightarrow\mathbb{R}\cup\{\infty\}$,  is defined by $\lambda(x)=\inf_{t\geq 0}\{ xt=x\}$. A point is said to be periodic provided $\lambda(x)\neq \infty.$}. Then $x$ is quasi-regular and $\lambda(x) A_x=[C_x]$, where $[C_x]$ represents the integral homological class determined by the orbit of $x$ under $\Phi$.
\end{lemma}

\begin{proof}
If $x$ is a fixed point, the assertion follows. Suppose  $x$ is not a fixed point. For each smooth function $f:M\longrightarrow \mathbb{R}$, it is easily seen that

\begin{align*}
\lim_{t\longrightarrow\infty}1/T\int_0^Tf(xt)dt=1/\lambda(x)\int_0^{\lambda(x)} f(xt)dt.
\end{align*}
Therefore every periodic point $x\in M$ is quasi-regular. Set $\alpha:[0,T]\longrightarrow M$ given by $\alpha(t)= xt$. Let $f\in C$. Then
	\[
A_x[f]=
\lim_{T\longrightarrow \infty}\frac{1}{2\pi i T}\int_{0}^{T}  \frac{f^{'}(xt)}{f(xt)}
	\]
	\[
=\lim_{T\longrightarrow \infty}\frac{1}{2\pi i T}\int_{0}^{T}  2\pi i\left(\omega_f(\alpha^{'}(t))dt\right)
	\]
	\[
=\lim_{T\longrightarrow \infty}\frac{1}{ T}\int_{\alpha([0,T])}  \omega_f
	\]
	\[
=1/\lambda(x)\int_{\alpha [0,\lambda(x)]} \omega_f.
	\]
 Parametrizing the orbit of $x$ by $\alpha$, and denoting by $[C_x]$   the integral homological class determined by the $x$-orbit, we have
	\[
[C_x]([f])= \int_{\alpha[0,\lambda(x)]}\omega_f=\lambda(x) A_x[f].
	\]
Therefore, 
	\[
\lambda(x) A_x= [C_x]
	.\]
\end{proof} 

\begin{theorem}\label{Schwartzman}
A smooth flow $\Phi$ on a closed manifold $M$ admits a cross-section if, and only if, for every invariant measure $\mu$ the homological class determined by $A_\mu$ is nonzero.
\end{theorem}

\begin{proof}[Proof's sketch] 
The set $\mathcal{C}=\{A_\mu; \mu \text{ invariant measure}\}$ is   convex and weakly compact. Since $\mathcal{B}_1$ (the boundary subspace) is  weakly-closed and $\mathcal{C}\cap \mathcal{B}_1=\emptyset$, by Hahn-Banach Theorem there exists a   continuous linear functional $\varphi: \mathcal{D}_1^{'}\longrightarrow\mathbb{R}$ such that 
	\[
\varphi\vert_{\mathcal{C}}>0\text{ and }\phi|_{\mathcal{B}_1}=0
	.\]
	Due to a famous theorem of Laurent Schwartz \cite{schwartzman1957asymptotic}, we know that the space of currents is reflexive. Hence, there exists an 1-form $\eta$ such that $\varphi=\eta.$ Therefore
	\[
d\eta(\phi)=\eta(\partial\phi)=\varphi(\partial\phi)=0,
	\]
concluding that $\eta $ is closed. The positivity condition of $\varphi$ on $\mathcal{C}$ implies that   
	\[
A_\mu(\eta)=\eta(A_\mu)=\varphi(A_\mu)>0
	\]
for every invariant measure $\mu$. Now, the set of all invariant measures is weakly compact.  By  Tischler's argument \cite{tischler41fibering}, $\eta$ can be approximated in $\mathcal{D}_1$ by forms with integral periods (we need a small change in the argument to obtain an approximation in $\mathcal{D}_1)$). It follows that there exists a form $\omega$ with integral periods such that $\omega(A_\mu)>0$ for all invariant measure $\mu$. Since $\omega=f^*d\theta$ for some smooth function $f:M\longrightarrow \mathbb{S}$, it follows from \cite[Theorem pag. 281]{schwartzman1957asymptotic} that $\Phi$ admits a global cross-section.
\end{proof}


\section{The class of $k$-forms of rank $k$}

Let $\omega$ be a $k$-form on a $n$-dimensional manifold $M$. We will denote by $\ker\omega$ the distribution on $M$ given by
    \[
x\rightarrow \ker\omega_x=\{v\in T_xM; i_v\omega=0\}.
    \]
The \textit{rank} of $\omega$ at a point $x\in M$ is defined by $r_x=n-\dim\ker_x$. A form $\omega$ is said to be constant rank if $r_x$ is independent of $x$. 
 Suppose $\omega$ be a closed form of rank $k$. The interior product with respect to the commutator of two vector fields satisfies the identity $i_{[X,Y]}=[\mathcal{L}_X,i_Y]$.  Given $X,Y\in\ker\omega$, we have
    \[
i_{[X,Y]}\omega=\mathcal{L}_X i_Y\omega - i_Y\mathcal{L}_X\omega=0,
        \]
 since $i_X\omega=i_Y\omega=0$, $d\omega=0$ and $\mathcal{L}_X=di_X+i_Xd$ (Cartan's magic formula). It follows from Frobenius' Theorem that $\ker\omega$ is tangent to a foliation $\mathcal{F}_\omega$ of codimension $k$.

\begin{theorem}\label{complementary foliation given by closed form} Let $M$ be a manifold with a closed $k$-form $\omega $ of rank $k$. Then $\omega$ is I.H. if, and only if, there exists a closed form $\eta$ of rank $q=n-k$ such that $\ker\omega  \cap \ker\eta  = \{0\}$.
\end{theorem}

\begin{proof} Let $\omega$ be a closed $k$-form of rank $k$. Suppose that there is on $M$ a Riemannian metric $g$ such that $\omega$ is $g$-harmonic. Then the form  $\eta := *\omega$ is a closed $q$-form with the desired property.
Conversely, let $\eta$ be a $q$-form with the property above. Choose Riemannian metrics $g_1,\ g_2$ on the bundles  $\ker\omega$\ and $\ker\eta$ and let $g$ be the Riemannian metric on $M$ which is the orthogonal sum of $g_1$ and $g_2$. The forms $*\omega$ and $\eta$ are volume forms in $\ker\eta^{\perp}$ and have the same nullity space. Hence
$$*\omega = s\eta$$
where $s: M \longrightarrow \mathbb{R}$ is a smooth function. We have
$$\omega \wedge *\omega =  \|\omega\|^2 \Omega,\ \ \ \ \omega \wedge \eta = r\Omega,$$
where $\Omega$ is the volume form of $g$, and $r: M \longrightarrow \mathbb{R}$\ is a smooth function. It follows that
	\[
\|\omega\|^2\Omega = \omega \wedge *\omega =\omega \wedge s \eta = sr\Omega.
	\]
Therefore $sr = \|\omega\|^2$. Let $\{E_i\} $ and $\{ F_j\}, i = 1, \ldots q,\ j =1, \ldots, k$\ be orthonormal bases for $\ker\omega,\ \ker\eta$, respectively, and let $f: M \longrightarrow \mathbb{R}$ be a positive function. Setting $E'_i
= f^{\frac{-1}{2}}E_i$, then $\{E'_i, F_j\}$\ is an orthonormal basis with respect to the metric tensor
	\[
g_f = fg_i \oplus g_2.
	\]
We look for a solution $f$ to
\begin{equation}\label{semiconformal}*_{g_f}\omega = \eta.
\end{equation}
If $f$\ is such a solution,\ $\omega$\ is harmonic with respect to $g_f$,\ since $\eta$\ is closed, and we have:

	\[
\omega(F_1, \ldots, F_k) = *_{g_f}\omega(E'_1, \ldots, E'_q) = \eta(E'_1, \ldots, E'_k) = f^{\frac{p-n}{2}}\eta(E_1, \ldots, E_k).	\]
Hence
	\[
f = (\omega(F_, \ldots, F_k)/\eta(E_1, \ldots, E_k))^{\frac{2}{k-n}} = s^{\frac{2}{k-n}}
	\]
is a solution of (\ref{semiconformal}).
\end{proof}

\begin{definition}\rm{
A $k$-form $\omega$ on $M$ is said to be \textit{transitive} at $x\in M$ if $x$ is contained in a closed embedded smooth $k$-dimensional submanifold $\Sigma \subset M$, to which $\omega$ restricts to a volume form\footnote{In other words, the restriction $\omega|_\Sigma$ is a top degree form without zeros. We usually denote this simply by writing $\omega|_\Sigma>0$.}. The form $\omega$ is said to be transitive if it is transitive at every $x\in M$ where it does not vanish.}
\end{definition}

\begin{lemma}(Calabi's argument)\label{Calabi's argument}
 Let $M$ be a manifold and $\omega$ be a  nowhere-vanishing $k$-form  on $M$. Suppose that every $x\in M$ for which  $\omega_x\neq 0$ is contained in a closed $k$-dimensional submanifold $\Sigma$ with normal trivial bundle and such that $\omega|_\Sigma>0$.  Then there exists a closed $(n-k)$ form $\psi$ such that $\omega\wedge\psi>0$ everywhere.
\end{lemma}

\begin{proof}
 This proof is a generalization of the argument used to the case of $1$ forms due  to Eugenio Calabi \cite{calabi1969intrinsic}. Let $x\in M$ be contained in a closed submanifold $\Sigma$, with normal trivial bundle. Let 
 	\[
\phi:\Sigma\times  \mathbb{D}^{n-k}\longrightarrow V
	 \]
be a diffeomorphism satisfying $\phi(x,0)=x$ for $x\in \Sigma$ (this can be obtained by trivializing the normal bundle of $\Sigma$). Denoting $\phi=(x,y)$, we have 
    \[
\omega_x(\partial x_1,\cdots,\partial x_p)>0
    \]
for $x\in \Sigma$. By compactness,  $\omega\wedge dy>0$  in an open set $V_0$ with $\Sigma\subset V_0\subset V$. Let $\epsilon>0$ such that $V_{2\epsilon}\subset V_0$ and $h:[0,2\epsilon]\longrightarrow \mathbb{R}$ be a smooth function satisfying

\[
\left\{\begin{array}{lll}
h(t)=1 \text{ if }t<\epsilon\\
h(t)=0 \text{ if }t\geq2\epsilon\\
h\geq 0 \text{ everwhere}.
\end{array}\right.
\]
The differential form
	\[
\eta=h(\Vert y\Vert^2)dy
	\]
is closed on $M$ and satisfies $\omega\wedge\eta \geq 0$, $\omega\wedge\eta>0$, where $\eta\neq 0$  and $\omega\wedge\eta=\omega\wedge dy>0$ in $V_\epsilon$. Consider an countable locally finite cover of $M$ by opens sets $\{V_i\}_{i\in I}$, with respective closed forms $\eta_i$ satisfying the three conditions above.
The expression
	\[
\psi=\sum_{i\in I}\eta_i	
	\] 
gives a well-defined closed differential form satisfying $\omega\wedge\eta>0$ in $M$. 
\end{proof}

\begin{remark}\label{transitive p-form of rank k admits transversal forms}
The condition about the normal bundle in the last Lemma can be weakened. Indeed, let $\omega$ be a transitive closed $k$-form of rank $k$. For each point $x\in M$ there exists a closed $k$-dimensional submanifold $\Sigma_x$ containing $x$ such that  $\omega\vert_{\Sigma_x}>0$. It follows that for each $\mathcal{F}_\omega$-holonomy invariant  measure $\mu$, there exists a submanifold transverse to $\mathcal{F}_\omega$ and contained in the support of $\mu$ (if $x$ is in the support of $\mu$, takes $\Sigma_x$). From \cite[Theorem 1.1]{plante1975foliations}, each such $\mu$ determines a nontrivial foliation cycle. It follows from \cite[Theorem I.7]{sullivan1976cycles} that there exists a closed form $\eta$ such that $\omega\wedge \eta$ determines a volume form on $M$, concluding the claim.  

In the intermediate degrees, considering a $k$-form of rank $k$, we cannot conclude from the volume property of $\omega\wedge\eta$   that  $\eta$ has rank equal to $(n-k)$. This illustrates the serious difficulties that we have when discussing forms of degrees strictly between 1 and
$(n-1)$, and clarify why Calabi's argument does not generalize to those forms (a concrete example of a transitive $2$-form on a $4$-manifold which is not I.H., due to Latschev, was presented by Volkov in \cite{volkov2008characterization}.)
\end{remark}


\section{An intrinsic characterization of harmonic nowhere-vanishing $(n-1)$-forms}

 In \cite{honda1997harmonic},    Honda characterized a class of $(n-1)$-forms, possessing only isolated zeros, as being I.H when they are transitive. Hence it is natural to expect that a nowhere-vanishing $(n-1)$-form is intrinsically harmonic, provided it admits a global cross-section, i.e., a closed $(n-1)$-dimensional submanifold everywhere transverse to the flow and cutting every orbit. This section addresses to show this fact. Initially, we present a preliminary characterization whose idea is essentially due to Calabi. 

\begin{theorem}\label{transitive nowhere-vanishing 1 or (n-1)-forms are I.H.}
Let $M$ be an orientable manifold. Suppose that $\omega$ is a nowhere-vanishing transitive closed 1-form (or $(n-1)$-form)  on $M$. Then $\omega$ is I.H.
\end{theorem}

\begin{proof}
Since  $M$ is an orientable manifold, any closed 1-dimensional or any closed orientable $(n-1)$-dimensional submanifold of $M$ have a normal trivial bundle in $M$. Thus, in the case that $\omega$ is a  transitive 1-form or a transitive $(n-1)$-form, we can apply Lemma \ref{Calabi's argument} to obtain a closed form $\psi$ defined on $M$ such that $\omega\wedge\psi>0$.  The proof is completed by applying Theorem  \ref{complementary foliation given by closed form} observing that, in  both cases, $\omega$ and $\psi$ have the correct rank.  
\end{proof}

The following theorems establish two criteria to decide when a closed nowhere-vanishing $(n-1)$-form is I.H.
 
\begin{theorem}\label{An intrinsic characterization of nowhere-vanishing harmonic (n-1)-forms}
Let $M$ be a closed orientable manifold. A nowhere-vanishing volume-preserving flow defined on $M$ admits global cross-section if and only if the induce closed nowhere-vanishing $(n-1)$-form is I.H.
\end{theorem} 

\begin{proof} 
Let $\omega$ be a closed nowhere-vanishing harmonic $(n-1)$-form defined on an orientable Riemaniann manifold $(M,g)$. Then $\omega\wedge *_g\omega$ is  a volume form on $M$, with $*_g\omega$ being a closed nowhere-vanishing 1-form. By Tischler's argument \cite{tischler41fibering}, there is a 1-form $\eta$ with integral periods such that $\omega\wedge\eta$ is a volume form on $M$. Each of such forms have the form $f^*d\theta$ for some smooth function $f:M\longrightarrow\mathbb{S}^1$, where $d\theta$ is the obvious volume form on $\mathbb{S}^1$. Since $\eta=f^*d\theta$ is nowhere-vanishing, the function $f$ is a submersion. By a well known lemma of Charles Ehresmann \cite{ehresmann1947topologie,earle1967foliations} (\textit{Ehresmann's  lemma}), $\mathcal{B}=\{M,f,\mathbb{S}^1, F\}$ defines a fiber bundle with fiber $F$ diffeomorphic to $f^{-1}(\theta)$, $\theta\in\mathbb{S}^1$. Since $\Phi_\omega$ induces a 1-dimensional foliation of $M$  transversal  to the  fibers of this fiber bundle, we have by a theorem of Charles  (\textit{Ehresmann's theorem}, \cite[Proposition 1, pag. 91]{camacho2013geometric})  that 
$\mathcal{B}=\{M,f,\mathbb{S}^1\}$ is a foliated bundle, the foliation given by the orbits of $\Phi_\omega$, of course. In particular, every fiber of this bundle intercept every orbit of $\Phi$, concluding that $f^{-1}(\theta)$ is a global cross-section to $\Phi_\omega$ for all $\theta\in\mathbb{S}^1$ (not necessarily connected).  

Conversely, suppose $\Phi_\omega$ admits a global cross-section $\Sigma$. We can assume $\omega|_\Sigma>0$. By a standard argument from foliation theory (sliding $\Sigma$ transversely in the flow direction), it follows $\Phi_\omega$ has a global cross-section at every point $x\in M$. Thus $\omega$ is transitive. From Theorem \ref{transitive nowhere-vanishing 1 or (n-1)-forms are I.H.}, we conclude that $\omega$ is I.H. 
\end{proof} 

\begin{remark}\label{examples of flows with C^1 foliations without global cross-section}\rm{
The geodesic flow of a Riemannian manifold with negative sectional curvature gives us an example of a nowhere-vanishing volume-preserving flow admitting complementary $C^1$ foliation (is an Anosov flow). However, the inducing closed nowhere-vanishing $(n-1)$-form cannot be I.H. In this case, we cannot obtain a transversal foliation induced by the kernel of a closed 1-form.} 
\end{remark}

\begin{lemma}\label{fileds with complementary foliations}
Let $X$ be a nowhere-vanishing $C^r$-vector field admitting a complementary $C^r$-foliation $\mathcal{F}$, $r\geq 2$. Then,  $\mathcal{F}=\mathcal{F}_\eta$ for some closed  1-form $\eta$ of class $C^{r-1}$.
\end{lemma}

\begin{proof}
The foliation $\mathcal{F}$ is transversely orientable since it is transversal to the nowhere-vanishing vector field $X$. It is easily seen that there exists a differential $C^{r-1}$-form $\omega$ such that $T\mathcal{F}=\ker\omega$.  We claim that $\eta=\frac{1}{\omega(X)}\omega$ is closed. Consider  $(x,y)$ and $(u,z)$ be  local coordinates with common domain such that  $X=\partial_x$ and $\mathcal{F}$ is given locally by  $z=$constant. By transversality, $(x,u)$ defines a coordinate neighborhood. In this coordinate, since $\partial_{u_i}$ is tangent to $\mathcal{F}$, $\eta(X)=1$ and $[\partial_x,\partial_{u_i}]=0$, then
	\[
d\eta(\partial_x,\partial_{u_i})=
\partial_{x}(\eta(\partial_{u_i}))-\partial_{u_i}(\eta(\partial_x))+\omega[\partial_x,\partial_{z_i}]=0
	.\]
It follows that $\eta$ is closed. Now, note that $\ker\eta=\ker\omega$, hence $\eta$ is a closed form  and $\mathcal{F}=\mathcal{F}_\eta$. 

\end{proof}

\begin{theorem}\label{An intrinsic characterization of nowhere-vanishing harmonic (n-1)-forms by foliation}
Let $M$ be a closed orientable manifold. A nowhere-vanishing $C^r$-volume-preserving flow on $M$ admits a $C^r$-global cross-section if and only if it admits a $C^r$-transversal foliation ($r\geq 2$).
\end{theorem}

\begin{proof} 
Let $\Omega$ be a volume form on $M$ and $\Phi$ be a flow preserving it. Let $X$ be a vector field generating  $\Phi$. Set $\omega=\iota_X\Omega$. If $X$ admits a complementary foliation, by Lemma \ref{fileds with complementary foliations} there exists a closed 1-form $\eta$ such that $\eta(X)=1$. Thus, $\omega\wedge\eta>0$ with $d\omega=0$ and $d\eta=0$. Hence, Theorem   \ref{complementary foliation given by closed form}  applies, concluding that $\omega$ is I.H. Therefore, by Theorem \ref{An intrinsic characterization of nowhere-vanishing harmonic (n-1)-forms}, $\Phi$ admits a global cross-section. Conversely, if $\Phi$ admits a global cross-section, then again by Theorem \ref{An intrinsic characterization of nowhere-vanishing harmonic (n-1)-forms}  there is a Riemannian metric $g$ on $M$ such that $*_g\omega$ is a closed 1-form. The foliation generated by the distribution $\ker *_g\omega$ is transversal to $\Phi$. The proof is completed.  
\end{proof}

\begin{remark}
From a result of Plante (\cite{plante1972anosov}, corollary 2.11), the latter theorem follows for  $C^1$ flows admitting complementary foliation given by the kernel of a nowhere-vanishing continuous closed 1-form. This result is the main ingredient in the preprint \cite{schwartzman1957asymptotic}, which states the same one as Theorem \ref{An intrinsic characterization of nowhere-vanishing harmonic (n-1)-forms}. It is worth pointing out that the characterization of I.H. nowhere-vanishing $(n-1)$-forms is implicit in Honda's thesis. Our proof here follows the spirit of Calabi's work and does not use Plante's result above.
\end{remark}

\section{A class of intrinsically harmonic forms}

This section was originated in an attempt to provide an example of a nowhere-vanishing closed $(n-1)$-form, non-exact and non I.H. Unfortunately, we did not get such an example nor manage to show that obtaining it is impossible. 

We will describe our first attempt and the results from them. The start was searching in the class of forms given in circle bundles $\mathcal{B}=\{B,p,\mathbb{T}^{n-1},\mathbb{S}^1\}$. By Tischler's argument and characterization of the covering of the torus, if the form $p^*\Omega_{\mathbb{T}^{n-1}}$ is I.H., then one can prove that $B$  and $\mathbb{T}^n$ are diffeomorphic. Hence, a possible counterexample would arise from a circle bundle $p:B\longrightarrow \mathbb{T}^{n-1}$  with total space not diffeomorphic to the $n$-torus and $p^*_{\Omega_{\mathbb{T}^{n-1}}}$ a non-exact form. However, we conclude that this attempt fails in any circle bundle, not only in circle bundles over the torus. 

Next, we will present a class of intrinsically harmonic forms. To this aim,  let us mention a result of Montgomery and, for later use, a mild generalization of it  \cite[Theorem 7.3]{epstein1976foliations}.

\begin{theorem}(Montgomery, \cite{montgomery1937pointwise})\label{Montgomery}Let $X$ be a connected metric space locally homeomorphic to $\mathbb{R}^n$.  If $T$ is a pointwise periodic homeomorphism of $X$,
then $T$ is periodic.
\end{theorem}

\begin{theorem}(K.C. Millett)\label{Montgomery generalization}
Let $M$ be a connected manifold and let $G$ be a group of homeomorphisms of $M$, such that the orbit of any point of $P$ is finite. Then $G$ is finite.
\end{theorem}


\begin{theorem}\label{examples of I.H. forms}
Let $M$ be a  closed smooth manifold  with a closed $(n-1)$-form $\omega$ that induces a pointwise periodic flow $\Phi$. If each orbit of the flow induced by $\omega$ is homologous to each other and $[\omega]\neq 0$, then $\omega$ is I.H. If $\omega$ is I.H., then there exists a smooth $\mathbb{S}^1$-action on $M$ with the same orbits as the one from the flow induced by $\omega$.
\end{theorem}

\begin{proof} Let $\Omega$ any volume form on $M$ and $X$ vector field given implicitly in the equation   $\omega=i_X\Omega$. By hypothesis, every orbit under $\Phi_X$ determines the same element in homology (of course, do not matter the volume form chosen initially). Let $\mu$ the measure induced by $\Omega$. By Lemma \ref{canonical asymptotic cycle associated with a volume-preserving flow},  $A_\mu=[\omega]$. By  De Rham's theory of currents, the homological class of $\omega$ as a diffuse current is zero if and only if the cohomology class of $\omega$ is zero. It follows that $A_\mu \neq 0$. Now,  given $f\in C/R$ and $\nu$ be a $\Phi$-invariant measure, it follows from Lemma \ref{asymptotic cycles of periodic points} that
\begin{align*}
A_\nu[f ]=\int_MA_x[f ]d\nu(x)\\
=\int_M\frac{1}{\tau_x}([C_x],[f ])d\nu(x).\\
\end{align*}
Therefore, $[C_x]\neq 0$ for all $x\in M$, since $[C_x]$ independ on $x$ by hypothesis and $A_\mu\neq 0$. Let  $x_0\in M$ and $f\in C/R$ such that $([C_{x_0}],[f])\neq 0$. Then
	\[
A_\nu[f ]=([C_{x_0}],[f])\int_M\frac{1}{\lambda(x)}d\nu(x)
	\]
for all $\Phi_X$-invariant measure $\nu$. Now, since $x\longrightarrow \lambda(x)$ is a positive function\footnote{The function $\lambda$ is lower-semicontinuous, hence Lebesgue measurable. It follows that $\int_Mfd\nu\neq 0$ for any measure $\nu$ defined on the $\sigma$-algebra generated by all Borel subsets of $ M$.}, we conclude that $A_\nu\neq 0$ for all $\Phi_X$-invariant measure $\nu$. Theorems  \ref{An intrinsic characterization of nowhere-vanishing harmonic (n-1)-forms by foliation} and  \ref{Schwartzman} together implies that $\omega$ is I.H.

 Conversely, suppose that $\omega$ is I.H. and let $X, \Omega$ with $\omega=i_X\Omega$.  Let $\eta$ be a closed 1-form with $\omega\wedge \eta>0$. Choose Riemannian metrics $g_1, g_2$  on the bundles generated by $X$ and $\ker \eta$, with $g_1(X,X)=1$, and let $g$ be the Riemannian metric on $M$ which is the orthogonal sum of $g_1$ and $g_2$.  It is easily to check that $\mathcal{L}_Xg=0$. Hence $X$ is a Killing vector field concerning to $g$ \cite[Proposition 26]{petersen2006riemannian}. From this we concludes that $\Phi_t$, the flow induced by $X$, is a isometry for all $t$ and, therefore, it is a isometry concerning the natural metric on $M$ induced by $g$. It follows by \cite[Theorem 1]{francaelizeu2022} that $\Phi$ is periodic.  It is easily seen that a periodic flow is induced by a circle action with the same flow's regularity, concluding the proof. 
\end{proof}


\section{Characterization of flat circle bundles}

It is well-known that every orientable smooth circle bundle  $\mathcal{B}=\{B,p,M,\mathbb{S}^1\}$ admits principal bundle structure. On other hand, a principal bundle is trivial if, and only if, it admits a global cross-section, i.e., a smooth function $s:M\longrightarrow B$ satisfying $ps=1$. The attempt to construct a global cross-section concerning a circle bundle, furnish an obstruction, to known, a element
	\[
\chi(\mathcal{B})\in H^2(M;\mathbb{Z})
	\]  
named \textit{Euler characteristic} of the bundle $\mathcal{B}$ (\cite[page 346]{morita2001geometry}). Hence, an orientable circle bundle  $\mathcal{B}$ has a global cross-section if and only if $\chi(\mathcal{B})=0$. This cohomology class carries not only information about the triviality of the bundle. As we will see, it is possible to characterize when is a circle bundle smooth foliated in terms of this class.

There is a natural bijection between the set of all smooth foliated orientable circle bundles over $M$ modulo leaf preserving bundle isomorphism and the set of all homomorphisms $\varphi:\pi_1(M)\longrightarrow \text{Diff}^+(\mathbb{S}^1)$ modulo conjugacy.  Many papers were concerned about the existence of a codimension-one foliation transverse to the fibers of a given circle bundle, i.e., when is foliated a circle bundle.
About circle bundles over surfaces, a necessary and sufficient condition for the existence of a transverse foliation was obtained by Milnor \cite{milnor1958existence} and Wood  \cite{wood1971bundles}. Circle bundles over a three-manifold were studied by  Miyoshi \cite{miyoshi1997foliated}. An answer to when is smooth foliated a circle bundle was given in \cite{miyoshi2001remark,oprea2005flat}. 

In this section, the main result is to characterize a flat circle bundle over an orientable base by proving Theorem \ref{main result}. Also is presented similar to the non-orientable cases. Most of the claims in Theorem \ref{main result} are known in the literature. However, we furnish alternative proof based on our techniques. We began with a useful lemma. 
 
\begin{lemma}\label{fiber lemma}
Let $\mathcal{B}=\{B,p,M,F\}$ be a differentiable fiber bundle with compact total space and base be an orientable manifold. Let $\Omega$ be a volume form on $M$ with $\int_M\Omega=1$.  Then, the Poincaré dual to the fiber of $\mathcal{B}$ can be represented by $p^*\Omega $. In particular, $[F]=0 $ in $H_{\dim F}(B;\mathbb{R})$ if and only if $p^*\Omega$  is an exact form.
\end{lemma}

\begin{proof}
Let $x\in M$. Denote by $F$ the fiber over $x$. Let $U$ be a sufficiently small   open set diffeomorphic to   $\mathbb{R}^n$ such that there exists a trivialization
  \[
\phi:p^{-1}(U)\longrightarrow U\times F
,\]
for which the normal bundle of $F$ is equivalent to the trivial vector bundle 
    \[
\tau:F\times U\longrightarrow F.
    \]
Let $\pi_1:U\times F\longrightarrow U$
the projection on the first factor. Given a $n$-form $\omega$ in $U$ generating  $H^n_c(U)=\mathbb{R}$ (the compactly supported cohomology), the closed form
	\[
\eta=\pi_1^*\omega
	\]
represents the \textit{Thom class} of the bundle $\tau$, because generates  $H^n_c(\{y\}\times U)$ for each $y\in F$ \cite[proposition 6.18]{bott1982differential}. On the other hand, the Thom class of  $\tau$ can be represented by the Poincaré dual of the null section, which in this case is $F$ \cite[proposition 6.24]{bott1982differential}. Then, $\eta$ represents the Poincaré dual of $F$ in $U\times F$ and, therefore,
	\[
\phi^*\eta=\phi^*(\pi_1^*\omega)=(\pi_1\circ\phi)^*\omega=p^*\omega
	\]
represents the Poincaré dual of $F$ in  $p^{-1}(U)$ (since $\phi$ is a orientation preserving-diffeomorphism \cite[page 69]{bott1982differential}). The trivial extension of  $p^*\omega$ to $B$  gives the Poincaré dual of $F$ in $B$ (\textit{localization principle}). Clearly, by a trivial extension, $\omega$ defines a form  in $M$  and
	\[
([\omega],[M])=\int_M\omega=\int_U\omega=1=([\Omega_M],[M])
	,\]
concluding that $\Omega_M $ and $\omega$ determines the same cohomology class. To finish, $p^*\Omega_M$ and $p^*\omega$ determines the same cohomology class with  $\pi^*\omega$ representing the Poincaré dual of $F$ in $B$. Therefore, we have $[F]=0$ in $H_{\dim F}(B;\mathbb{R})$ if, and only if, $p^*\Omega_M$ is an exact form.
\end{proof}

As one example,  let $\mathcal{B}=\{B,p,M,F\}$ be a differentiable fiber bundle with compact total space and base be an orientable manifold. Suppose that $\chi(F)\neq 0$ (the Euler characteristic of $F$ is nonzero). Let $\Omega$ be a volume form on $M$. Then $\pi^*\Omega$ determines a nonzero cohomology class by \ref{fiber lemma}. Indeed, since $\chi(F)\neq 0$, it is not possible for $F$ to be a boundary and, therefore, we have $[F]\neq 0$ in $H_{\dim F}(B;\mathbb{R})$.

\begin{theorem}\label{The main theorem 1}
Let $\mathcal{B}=\{B,p,M,\mathbb{S}^1\}$ be a smooth orientable circle bundle over a closed orientable manifold. If $p^*(\Omega_M)$ determines a nontrivial cohomology class, then $\mathcal{B}$ is $\text{Diff}^+(\mathbb{S}^1)$-equivalent to a flat circle bundle.
\end{theorem}

\begin{proof}
By Lemma \ref{fiber lemma}, the $\mathcal{B}$'s fibers are homologous to each other. Since the flow generated by $p^*(\Omega_M)$ has the fibers of $\mathcal{B}$ as its orbits, it follows by Theorem \ref{examples of I.H. forms} that $p^*(\Omega_M)$ is I.H.  Then, there exists a closed 1-form $\eta$ transversal to the flow generated by $p^*(\Omega_M)$. By Tischler's argument \cite{tischler41fibering}, we can take $\eta$ with integral periods. Let $f:B\longrightarrow \mathbb{S}^1$ be a smooth function such that  $\eta=f^*d\theta$. It follows that $\mathcal{B}$ is foliated by compact leaves. By fiber's compactness, this bundle become a foliated bundle \cite[Chapter IV, \S 2]{camacho2013geometric}. Let 
	\[
\varphi:\pi_1(M)\longrightarrow \text{Diff}^+(\mathbb{S}^1)
	\]
be the holonomy homomorphism which characterizes this foliated bundle. Since each leaf of $\mathcal{F}_\eta$ is  compact and transversal to the fibers of $\mathcal{B}$, then the orbit of any point  in the fiber $B_x$ over $x$ under the group $\varphi(\pi_1(M))$  is finite (indeed, it is equal to $\# (f^{-1}(f(x))\cap B_x)$). This fact and Theorem \ref{Montgomery generalization} imply that $\varphi(\pi_1(M))$ is a finite group. Now, let $g$ be a Riemannian metric on $\mathbb{S}^1$. Set
	\[
h(X,Y)=\frac{1}{|\varphi(\pi_1(M))|}\sum_{k\in\varphi(\pi_1(M))}g(d\varphi(k)(X),d\varphi(k)(Y)).
	\]
Then, $h$ is a Riemannian metric on $\mathbb{S}^1$ such that $\varphi(k)$ is an isometry for all $k\in \text{Im}(\varphi)$. By the well-known characterization of isometries of the circle, it follows that   there exists a diffeomorphism $f:\mathbb{S}^1\longrightarrow\mathbb{S}^1$ such that $f^{-1}\circ\varphi(\alpha)\circ f\in\mathbb{S}^1$ for all $\alpha\in\pi_1(M)$.  The  fiber bundle obtained considering as holonomy homomorphism the map 
	\[
\varphi^{'}:\pi_1(M)\longrightarrow  \mathbb{S}^1
	\]
 given by $\varphi^{'}(\alpha)=f^{-1} \varphi(\alpha)f$ is a differentiable principal  flat bundle \cite[Lemma 1]{milnor1958existence}. Furthermore, this bundle is $\text{Diff}^+(\mathbb{S}^1)$-equivalent to $\mathcal{B}$. The proof is completed. 
\end{proof}

\begin{theorem}\label{characterization of flat circle bundle in the linguage of I.H. forms}
Let $\mathcal{B}=\{B,p,M,\mathbb{S}^1\}$ be a differentiable principal circle bundle with $M$ closed and orientable.  Then $\mathcal{B}$ admits a flat connection if and only if the form $p^*(\Omega_M)$  is I.H.
\end{theorem}

\begin{proof}
Suppose that $\pi^*\Omega_M$ is I.H. Then, by Theorem \ref{The main theorem 1} there exists a differentiable principal flat circle bundle which is $\text{Diff}^+(\mathbb{S}^1)$-equivalent to $\mathcal{B}$. It is well-known that the group of all orientation preserving diffeomorphism of the circle retracts onto $\mathbb{S}^1$. It follows that  if  $\mathcal{B}$ is differentiable isomorphic to $\mathcal{B}^{'}$ as  $\text{Diff}^+(\mathbb{S}^1)$-bundles then the same is true as $\mathbb{S}^1$-bundles. To finish, a flat connection in $\mathcal{B}^{'}$ can be transferred to a flat connection for $\mathcal{B}$ \cite[pag. 79]{kobayashi1963foundations}. The converse follows easily from Theorem \ref{The main theorem 1}. 
\end{proof}

To finish the proof of Theorem \ref{main result} we need the following observation.

\begin{remark}(Plante, \cite{plante1974foliations})\label{Plante's example}
Given a smooth orientable sphere bundle	 $\mathcal{B}=\{B,p,M,\mathbb{S}^{k}\}$   we have  the Gysin cohomology
sequence (real coefficients)
$$...\longrightarrow H^k(M) \longrightarrow^{p^*}   H^k(B) \longrightarrow  H^0(M)\longrightarrow^{\Psi}  H^{k+1}(M) \longrightarrow... $$
where $\Psi(1)\in H^{k+1}(M)$ is just the Euler's class of the bundle. Hence, the fiber is null-homologous in $H_k(E)$ if and only if $p*$ is surjective. This happens if, and only if, the Euler's class is nonzero. 
\end{remark}

We are now in a position to prove Theorem \ref{main result}.

\begin{theorem}\label{theorem about Euler's class}
 A differentiable principal circle bundle $\mathcal{B}=\{B,p,M,\mathbb{S}^1\}$ over a closed orientable base admits a flat connection, if and only if, one (and consequently all) of the four conditions below holds:
\begin{itemize}
\item[(1)] $p^*(\Omega_M)$ is I.H.;
\item[(2)] $p^*(\Omega_M)$ determines a nonzero cohomology class;
\item[(3)] the fiber represents a nonzero class in $H_1(B;\mathbb{R})$;
\item[(4)] $\chi(\mathcal{B})$, the Euler's class of $\mathcal{B}$, is a torsion element. 
\end{itemize}
\end{theorem}

\begin{proof}
Note we have the equivalence between item 2 and item 3, by Lemma \ref{fiber lemma}, and between item 3 and item 4, by Plante's observation \ref{Plante's example} (considering that the real Euler's class $\chi_\mathbb{R}(\mathcal{B})$ of $\mathcal{B}$ is equal to $\chi(\mathcal{B})\otimes\mathbb{R}$).  Since item 2 follows from item 1 by Hodge's theorem and 1 from 2 by Theorem \ref{examples of I.H. forms},  the proof is completed by using Theorem \ref{The main theorem 1}.
\end{proof}

For a general smooth circle bundle, we have the following theorem.

\begin{theorem}\label{smooth foliated circle bundle over non-orientable manifold}
Let $\mathcal{B}=\{B,p,M,\mathbb{S}^1\} $ be a circle bundle over a  closed manifold. If fiber determines a nonzero real homological class then $\mathcal{B}$ is smooth foliated.
\end{theorem}

\begin{proof}
Let $\pi:\tilde{B}\longrightarrow B$ the double cover of $B$. Since $p$ is a submersion, by  Eheresmann's lemma \cite{ehresmann1947topologie}, $\tilde{\mathcal{B}}=\{\tilde{B},p\circ\pi,M\}$ determines a fiber bundle with fibers given by the connected components of $\pi^{-1}(p^{-1}(x))$ (that are at most two).   Let $B_x$ the fiber of $\mathcal{B}$ over $x$. Suppose that $[B_x]\neq 0 $ in $H_1(B;\mathbb{R})$. Then, given fiber $\tilde{B}_x$ of $\tilde{\mathcal{B}}$ satisfying $\pi(\tilde{B}_x)= B_x$, we have $\pi_*([\tilde{B}_x])=[B_x]\neq 0$, and $[\tilde{B}_x]\neq 0$. It follows by Theorem \ref{main result} that $\tilde{\mathcal{B}}$ is smooth foliated. Let $\tilde{\mathcal{F}}$ be a transversal foliation to $\tilde{\mathcal{B}}$. Then, $\pi_*(T\tilde{\mathcal{F}})$ is an involutive constant rank distribution transversal to the fibers of $\mathcal{B}$, since   $p_*[X,Y]=[p_*X,p_*Y]$. Hence, it give a  transversal foliation $\mathcal{F}$ on $B$, concluding that $\mathcal{B}$ is smooth foliated.   
\end{proof}

\providecommand{\href}[2]{#2}

\end{document}